\documentclass[submission]{FPSAC2023}

\theoremstyle{definition}
\newtheorem{theorem}{Theorem}
\newtheorem{thm}{Theorem}
\newtheorem{lem}{Lemma}
\newtheorem{defn}[theorem]{Definition}

\newtheorem*{remark}{Remark}
\usepackage[mathscr]{euscript}
\usepackage{multicol}
\newcommand{\y}{{ \mathscr{Y} }}
\usepackage{varwidth}
\usepackage{lipsum}
\usepackage{bbm}
\usepackage{comment}

\title[]{Stable-Limit Non-symmetric Macdonald Functions in Type A}

\author[]{Milo James Bechtloff Weising \thanks{\href{mjbechtloffweising@ucdavis.edu}{mjbechtloffweising@ucdavis.edu}}\addressmark{1}}

\address{\addressmark{1}Department of Mathematics, University of California, Davis}  
\received{11/14/2022}

\abstract{We construct and study an explicit simultaneous $\mathscr{Y}$ eigenbasis of Ion and Wu's standard representation of the $^+$stable-limit double affine Hecke algebra for the limit Cherednik operators $\mathscr{Y}_i$. This basis arises as a generalization of Cherednik's non-symmetric Macdonald polynomials of type $GL_n$. We utilize links between $^+$stable-limit double affine Hecke algebra theory of Ion and Wu and the double Dyck path algebra of Carlsson and Mellit that arose in their proof of the Shuffle Conjecture. As a consequence, the spectral theory for the limit Cherednik operators is understood.}


\usepackage[backend=bibtex]{biblatex}
\begin{document}

\maketitle

\section{Introduction}
This is a copy of the author's FPSAC 2023 submission. For the sake of satisfying the page limit for FPSAC most of the proofs are either only given as sketches or not given at all. The longer version with complete details will appear soon and possibly replace this version. 

The Shuffle Conjecture, now the Shuffle Theorem \cite{CM_2015}, is a combinatorial statement regarding the Frobenius character, $\mathcal{F}_{R_n}$, of the diagonal coinvariant algebra $R_n$ which generalizes the coinvariant algebra arising from the geometry of flag varieties. The following explicit formula is due to Haiman  \cite{HHLRU-2003}:
    $$ \mathcal{F}_{R_n}(X;q,t) = (-1)^n \nabla e_n[X] $$
    where the operator $\nabla$ is an eigenoperator on symmetric functions prescribed by its action on the modified Macdonald symmetric functions as
    $$\nabla \widetilde{H}_{\mu} = \widetilde{H}_{\mu}[-1]\cdot  \widetilde{H}_{\mu} .$$
    The original conjecture of Haglund, Haiman, Loehr, Remmel, and Ulyanov states the following:
        
\begin{thm}[Shuffle Theorem]\cite{HHLRU-2003}

$$ (-1)^n \nabla e_n[X] = \sum_{\pi}\sum_{ w \in WP_{\pi}} t^{\text{area}(\pi)} q^{\text{dinv}(\pi,w)} x_{w}. $$

\end{thm}

In the above, $\pi$ ranges over the set of Dyck paths of length $n$ and $WP_{\pi}$ is the set of word parking functions corresponding to $\pi$. The values $area(\pi)$ and $dinv(\pi,w)$ are certain statistics corresponding to $\pi$ and $w \in WP_{\pi}$.

In \cite{CM_2015}, Carlsson and Mellit prove the Compositional Shuffle Conjecture, a generalization of the original Shuffle Conjecture. The authors construct and investigate a quiver path algebra, $\mathbb{A}_{q,t}$,  called the Double Dyck Path algebra. They construct a representation of $\mathbb{A}_{q,t}$, called the standard representation, built on certain mixed symmetric and non-symmetric polynomial algebras with actions from Demazure-Lusztig operators, Hall-Littlewood creation operators, and plethysms. The Compositional Shuffle Conjecture 
falls out after a rich understanding of the standard representation is developed. Later analysis done by Carlsson, Gorsky, and Mellit \cite{GCM_2017} showed that in fact $\mathbb{A}_{q,t}$ occurs naturally in the context of equivariant cohomology of Hilbert schemes. 

Recent work by Ion and Wu \cite{Ion_2022} has made progress in linking the work of Carlsson and Mellit on $\mathbb{A}_{q,t}$ to the representation theory of double affine Hecke algebras. Ion and Wu introduce the $^+$stable-limit double affine Hecke algebra $\mathscr{H}^{+}$ along with a representation $\mathscr{P}_{as}^{+}$ of $\mathscr{H}^{+}$ from which one can recover the standard $\mathbb{A}_{q,t}$ representation. The main obstruction in making a stable-limit theory for the double affine Hecke algebras is the lack of an inverse system of the double affine Hecke algebras in the traditional sense. Ion and Wu get around this obstruction by introducing a new notion of convergence (Defn. \ref{defn5}) for sequences of polynomials with increasing numbers of variables along with limit versions of the standard Cherednik operators defined by this convergence. 

Central to the study of the standard Cherednik operators are the non-symmetric Macdonald polynomials. The non-symmetric Macdonald polynomials in full generality were introduced first by Cherednik \cite{C_2001} in the context of proving the Macdonald constant-term conjecture. The introduction of the double affine Hecke algebra, along with the non-symmetric Macdonald polynomials by Cherednik, constituted a significant development in representation theory. They serve as a non-symmetric counterpart to the symmetric Macdonald polynomials introduced by Macdonald as a q,t-analog of Schur functions. Further, they give an orthogonal basis of the polynomial representation consisting of weight vectors for the Cherednik operators. In particular, the correct choice of symmetrization applied to a non-symmetric Macdonald polynomial will yield its symmetric counterpart. The type A symmetric Macdonald polynomials are a remarkable basis of symmetric polynomials simultaneously generalizing many other well studied bases which can be recovered by appropriate specializations of values for q and t. The aforementioned modified Macdonald functions $\widetilde{H}_{\mu}$ can be obtained via a plethystic transformation from the symmetric Macdonald polynomials in sufficiently many variables. The spectral theory of non-symmetric Macdonald polynomials is well understood using the combinatorics of affine Weyl groups.

It is natural to seek an asymptotic extension for the non-symmetric Macdonald polynomials following the methods of Ion and Wu. In particular, does the standard $\mathscr{H}^{+}$ representation $\mathscr{P}_{as}^{+}$ have a basis of weight vectors for the limit Cherednik operators $\y_i$? The main result, Theorem \ref{thm7}, of this paper answers this question in the affirmative.

The strategy for finding a basis of weight vectors for the limit Cherednik operators $\mathscr{Y}_i$ is the following. First, we show that the non-symmetric Macdonald polynomials have stable-limits in the sense that if we start with a composition $\mu$ and consider the compositions $\mu * 0^m$ for $m \geq 0$ then the corresponding sequence of non-symmetric Macdonald polynomials $E_{\mu * 0^m}$ converges to an element $\widetilde{E}_{\mu}$ of $\mathscr{P}_{as}^{+}$. Next, we show that these limits of non-symmetric Macdonald polynomials are $\mathscr{Y}$-weight vectors. Importantly, the newly constructed set of $\widetilde{E}_{\mu}$ do \textbf{\textit{not}} span $\mathscr{P}_{as}^{+}$. To fill in these gaps, the lowering operators $d_{-}$ from $\mathbb{A}_{q,t}$ are used to create enough $\y$ weight vectors to span $\mathscr{P}_{as}^{+}$. Finally, a symmetrization operator is used to show that the spanning set obtained from this process is actually a basis in Theorem \ref{thm7}.

Lemma \ref{lem1}, Theorem \ref{thm5}, and Lemma \ref{lem5} together give a description of the weights across all weight vectors in $\mathscr{P}_{as}^{+}$. The author would like to thank the FPSAC referees who alerted the author to an unpublished work of Ion and Wu which independently determines the same explicit description of these eigenvalues.

\section{Definitions and Notation}

\subsection{Double Affine Hecke Algebras in Type GL}

\begin{defn} \label{defn1}
Define the \textbf{\textit{double affine Hecke algebra}} $\mathscr{H}_n$ to be the $\mathbb{Q}(q,t)$-algebra generated by $T_1,\ldots,T_{n-1}$, $X_1^{\pm 1},\ldots,X_{n}^{\pm 1}$, and $Y_1^{\pm 1},\ldots,Y_n^{\pm 1}$ with the following relations:

\begin{multicols}{2}
\begin{itemize}
    \item [(i)] $(T_i -1)(T_i +t) = 0$,
    \item [] $T_iT_{i+1}T_i = T_{i+1}T_iT_{i+1}$,
    \item [] $T_iT_j = T_jT_i$, $|i-j|>1$,
    \item [(ii)] $T_i^{-1}X_iT_i^{-1} = t^{-1}X_{i+1}$,
    \item []$T_iX_j = X_jT_i$, $i \notin \{j,j+1\}$,
    \item []$X_iX_j = X_jX_i$,
    \item [(iii)]$T_iY_iT_i = tY_{i+1}$,
    \item []$T_iY_j = Y_jT_i$, $i\notin \{j,j+1\}$,
    \item []$Y_iY_j = Y_jY_i$,
    \item [(iv)]$Y_1T_1X_1 = X_2Y_1T_1$,
    \item [(v)]$Y_1X_1\cdots X_n = qX_1\cdots X_nY_1$
    \item []
\end{itemize}
\end{multicols}

Further, define the special element $\omega_n$ by $$\omega_n := T_{n-1}^{-1}\cdots T_1^{-1}Y_1^{-1}$$

\end{defn}

\subsubsection{Standard DAHA representation}

\begin{defn} \label{defn2}
Let $\mathscr{P}_n = \mathbb{Q}(q,t)[x_1^{\pm 1},\ldots,x_n^{\pm 1}]$. The \textbf{\textit{ standard representation of $\mathscr{H}_n$}} is given by the following action on $\mathscr{P}_n$:

\begin{center}
\begin{varwidth}{\textwidth}
\begin{itemize}
    \item $T_if(x_1,\ldots,x_n) = s_i f(x_1,\ldots,x_n) +(1-t)x_i \frac{1-s_i}{x_i-x_{i+1}}f(x_1,\ldots,x_n)$
    \item $X_if(x_1,..,x_n)= x_if(x_1,\ldots,x_n)$
    \item $\omega_nf(x_1,\ldots,x_n) = f(q^{-1}x_n,x_1,\ldots,x_{n-1})$
\end{itemize}
\end{varwidth}
\end{center}

Here $s_i$ denotes the operator that swaps the variables $x_i$ and $x_{i+1}$. Under this action the $T_i$ operators are known as the \textbf{\textit{Demazure-Lusztig operators}}. For q,t generic $\mathscr{P}_n$ is known to be a faithful representation of $\mathscr{H}_n$. The action of the elements $Y_1,\ldots,Y_n \in \mathscr{H}_n$ are called \textbf{\textit{Cherednik operators}}.
\end{defn}

Set $\mathscr{H}_n^{+}$ to be the positive part of $\mathscr{H}_n$ i.e. the subalgebra generated by $T_1,\ldots,T_{n-1}$, $X_1,\ldots,X_n$, and $Y_1,\ldots,Y_n$ without allowing for inverses in the $X$ and $Y$ elements and set $\mathscr{P}_n^{+} = \mathbb{Q}(q,t)[x_1,\ldots,x_n]$. Importantly, 
$\mathscr{P}_n^{+}$ is a $\mathscr{H}_n^{+}$ submodule of $\mathscr{P}_n$. 

\subsubsection{Non-symmetric Macdonald Polynomials and Symmetric Functions}

\begin{defn} \label{defn3}
The \textbf{\textit{non-symmetric Macdonald polynomials}} (for $GL_n$) are a family of Laurent polynomials $E_{\mu} \in \mathscr{P}_n$ for $\mu \in \mathbb{Z}^n$ uniquely determined by the following:

\begin{itemize}
    \item Triangularity: Each $E_{\mu}$ has a monomial expansion of the form $E_{\mu} = x^{\mu} + \sum_{\lambda < \mu} a_{\lambda}x^{\lambda}$
    where $"<"$ denotes the Bruhat order for $\mathbb{Z}^n$
    
    \item Weight Vector: Each  $E_{\mu}$ is a weight vector for the operators $Y_1,\ldots,Y_n \in \mathscr{H}_n$.
\end{itemize}
\end{defn}

The non-symmetric Macdonald polynomials are a $Y$ weight basis for the $\mathscr{H}_n$ standard representation $\mathscr{P}_n$. For $\mu \in \mathbb{Z}^n$, $E_{\mu}$ is homogeneous with degree $\mu_1+\cdots +\mu_n$. Further, the set of $E_{\mu}$ corresponding to $\mu \in \mathbb{Z}_{\geq 0}^n$ gives a basis for $\mathscr{P}_n ^{+}$.

\begin{defn}
 In this paper, a \textbf{\textit{composition}} will refer to a finite tuple $\mu = (\mu_1,\ldots,\mu_n)$ of non-negative integers. We allow for the empty composition $\emptyset$ with no parts. The length of a composition $\mu = (\mu_1,\ldots,\mu_n)$ is $\ell(\mu) = n$ and the size of the composition is $| \mu | = \mu_1+\ldots+\mu_n$. Given two compositions $\mu = (\mu_1,\ldots,\mu_n)$ and $\beta = (\beta_1,\ldots,\beta_m)$, define $\mu * \beta = (\mu_1,\ldots,\mu_n,\beta_1,\ldots,\beta_m)$. A \textbf{\textit{partition}} is a composition $\lambda = (\lambda_1,\ldots,\lambda_n)$ with $\lambda_1\geq \ldots \geq \lambda_n \geq 1$. We denote $sort(\mu)$ to be the partition obtained by ordering the nonzero elements of $\mu$ in weakly decreasing order. 
Define the \textbf{\textit{ring of symmetric functions}} $\Lambda$ to be the inverse limit of the symmetric polynomial rings $\mathbb{Q}(q,t)[x_1,\ldots,x_n]^{S_n}$ with respect to the quotient maps sending $x_n \rightarrow 0$.
In this paper we use plethystic notation. For a complete introduction and explanation of plethysm we refer the reader to \cite{Stanley}. For example, if $F \in \Lambda$ and $\{t_1,t_2,\ldots\}$ is a set of independent variables, then we write $F[t_1+t_2+\cdots]$ for the symmetric function given by F with variables in the set $\{ t_1,t_2,\ldots\}$. We will in a few instances use the notation $\mathbbm{1}(p)$ to denote the value $1$ if the statement p is true and $0$ otherwise. 

\end{defn}

\subsection{Stable-Limit DAHA of Ion and Wu}
 
\begin{defn} \label{defn4}
The \textit{\textbf{$^+$stable-limit double affine Hecke algebra}} of Ion and Wu, $\mathscr{H}^{+}$, is the algebra  generated over $\mathbb{Q}(q,t)$ by the elements $T_i,X_i,Y_i$ for $i \in \mathbb{N}$ satisfying the following relations:

\begin{center}
\begin{varwidth}{\textwidth}
\begin{itemize}
    \item The generators $T_i,X_i$ for $i \in \mathbb{N}$ satisfy (i) and (ii) of Defn. \ref{defn1}.
    \item The generators $T_i,Y_i$ for $i \in \mathbb{N}$ satisfy (i) and (iii) of Defn. \ref{defn1}.
    \item  $Y_1T_1X_1 = X_2Y_1T_1$
 \end{itemize}
 \end{varwidth}
 \end{center}
 
 \end{defn}

We include Ion and Wu's full definition of convergence in Defn. \ref{defn5} for the sake of completeness. A full understanding of convergence is not required to follow  the rest of this paper. 

\begin{defn}\cite{Ion_2022} \label{defn5}
 Let $\mathscr{P}(k)^{+} := \mathbb{Q}(q,t)[x_1,\ldots,x_k]\otimes \Lambda[x_{k+1}+x_{k+2}+\ldots]$. Define the \textbf{\textit{ring of almost symmetric functions}} $\mathscr{P}_{as}^{+} := \bigcup_{k\geq 0} \mathscr{P}(k)^{+}$. Further, let $\mathscr{P}_{\infty}^{+}$ denote the inverse limit of the rings $\mathscr{P}_{k}^{+}$ with respect to the homomorphisms which send $x_{k+1}$ to 0 at each step. Note $\mathscr{P}_{as}^{+} \subset \mathscr{P}_{\infty}^{+}.$ Define $\rho: \mathscr{P}_{as}^{+} \rightarrow x_1\mathscr{P}_{as}^{+}$ to be the linear map defined by $\rho(x_1^{a_1}\cdots x_n^{a_n}F[x_{m}+x_{m+1}+\ldots]) = \mathbbm{1}(a_1 > 0) x_1^{a_1}\cdots x_n^{a_n}F[x_{m}+x_{m+1}+\ldots] $ for $F \in \Lambda$. Let $(f_k)_{k \geq 1}$ be a sequence of polynomials with $f_k \in \mathscr{P}_k^{+}$. Then the sequence $(f_k)_{k \geq 1}$ is \textbf{\textit{convergent}} if there exist some N and auxiliary sequences $(h_k)_{k\geq1}$, $(g^{(i)}_k)_{k\geq1}$, and $(a^{(i)}_k)_{k\geq 1}$ for $1\leq i \leq N$ with $h_k, g^{(i)}_k \in \mathscr{P}_{k}^{+}$, $a^{(i)}_k \in \mathbb{Q}(q,t)$ with the following properties:

\begin{itemize}
    \item For all k, $f_k = h_k + \sum_{i=1}^{N} a^{(i)}_k g^{(i)}_k$.
    \item The sequences $(h_k)_{k\geq1}$, $(g^{(i)}_k)_{k\geq1}$ for $1\leq i \leq N$ converge in $\mathscr{P}_{\infty}^{+}$ with limits $h,g^{(i)}$ respectively. Further, $g^{(i)} \in \mathscr{P}_{as}^{+}$.
    \item The sequences $a^{(i)}_k$ for $1\leq i \leq N$ converge with respect to the t-adic topology on $\mathbb{Q}(q,t)$ with limits $a^{(i)}$ which are required to be in $\mathbb{Q}(q,t)$.
\end{itemize}

The sequence is said to have a limit given by
$\lim_{k} f_k = h + \sum_{i=1}^{N}a^{(i)}g^{(i)}.$

\end{defn}

Ion and Wu use their definition of convergence to define asymptotic versions of the Cherednik operators.

\begin{thm}\cite{Ion_2022}
 Consider the sequence of operators $\widetilde{Y}_1^{(n)} := t^n \rho \circ Y_1^{(n)}$ where $Y_1^{(n)}$ is the operator coming from the action of $Y_1 \in \mathscr{H}_n^{+}$ on $\mathscr{P}_{n}^{+}$. Let  $\pi_n: \mathscr{P}_{as}^{+} \rightarrow \mathscr{P}_{n}^{+}$ be the canonical projection and let $f \in \mathscr{P}_{as}^{+}$. Then the sequence $(\widetilde{Y}_1^{(n)}\circ \pi_n(f))_{n \geq 1}$ is convergent with limit which is also almost symmetric. This yields a well-defined operator $\mathscr{Y}_1: \mathscr{P}_{as}^{+} \rightarrow \mathscr{P}_{as}^{+}$ given by $\mathscr{Y}_1(f) := \lim_n \widetilde{Y}_1^{(n)}\circ \pi_n(f)$. Further, the operator $\mathscr{Y}_1$ along with the Demazure-Lusztig action of the $T_i$'s and multiplication by the $X_i$'s generate an $\mathscr{H}^{+}$ action on $\mathscr{P}_{as}^{+}$.
\end{thm}

\section{Stable-Limits of Non-symmetric Macdonald Polynomials}

Given a composition $\mu$, consider the compositions $\mu * 0^m$ for $m\geq 0$ and the corresponding sequence of non-symmetric Macdonald polynomials $(E_{\mu*0^m})_{m \geq 0}$.
In order to prove the convergence of these sequences we use the following result of \cite{Haglund_2004} giving an explicit combinatorial formula for the non-symmetric Macdonald polynomials. Note that the $q,t$ conventions in \cite{Haglund_2004} differ from those appearing in this paper. In the below theorem the appropriate translation $q \rightarrow q^{-1}$ has been made. 

\begin{thm}\cite{Haglund_2004} \label{HHL}
For a composition $\mu$ with $\ell(\mu) = n$ the following holds:
    $$E_{\mu} = \sum_{\substack{\sigma: \mu \rightarrow [n]\\ \text{non-attacking}}} X^{\sigma}q^{-maj(\hat{\sigma})}t^{coinv(\hat{\sigma})} \prod_{\substack{u \in dg'(\mu) \\ \hat{\sigma}(u) \neq \hat{\sigma}(d(u))}} \left( \frac{1-t}{1-q^{-(\ell(u)+1)}t^{(a(u)+1)}} \right) $$
\end{thm}

The combinatorial description of non-symmetric Macdonald polynomials in the Haiman-Haglund-Loehr formula relies on the combinatorics of \textbf{\textit{non-attacking labellings}} of certain box diagrams corresponding to compositions. In the interest of space we refer the reader to \cite{Haglund_2004} for all the notation used above such as $\hat{\sigma}$, $d$, $a$, $\ell$, $maj$, and $coinv$.

We now show the convergence for the sequence $(E_{\mu*0^m})_{m \geq 0}$. The method used shows convergence and gives an explicit combinatorial formula for the limit functions. 

\begin{thm} \label{thm4}

For a composition $\mu$ with $\ell(\mu) = n$ the sequence $(E_{\mu*0^m})_{m \geq 0}$ is convergent with limit $\widetilde{E}_{\mu}$ in $\mathcal{P}_{as}^{+}$ given by 
$$ \widetilde{E}_{\mu} := \sum_{\substack{\lambda ~ \text{partition}  \\ |\lambda| \leq |\mu|}} m_{\lambda}[x_{n+1}+\cdots] \sum_{\substack{\sigma:\mu *0^{\ell(\lambda)} \rightarrow \{1,\ldots,n+\ell(\lambda)\}\\ \text{non-attacking} \\ |\sigma^{-1}(n+i)|= \lambda_i}} x_1 ^{|\sigma^{-1}(1)|}\cdots x_n ^{|\sigma^{-1}(n)|} q^{-maj(\hat{\sigma})}t^{coinv(\hat{\sigma})} \widetilde{\Gamma}(\hat{\sigma})$$  

where 

$$\widetilde{\Gamma}(\hat{\sigma}) = \prod_{\substack{ u \in dg'(\mu* 0^{\ell(\lambda)}) \\ \hat{\sigma}(u) \neq \hat{\sigma}(d(u)) \\ u ~ \text{not in row } 1 }} \left( \frac{1-t}{1-q^{-(\ell(u)+1)}t^{(a(u)+1)}} \right) \prod_{\substack{ u \in dg'(\mu *0^{\ell(\lambda)}) \\ \hat{\sigma}(u) \neq \hat{\sigma}(d(u)) \\ u ~ \text{in row } 1 }} \left( 1-t \right) $$

\end{thm}

\begin{proof}[Proof Sketch]
Start by using the HHL formula to expand $E_{\mu *0^m}$ for $m \geq 1$. Because $E_{\mu *0^m}$ is symmetric in $x_{n+1},\ldots,x_{n+m}$ we can expand relative to the monomial symmetric functions $m_{\lambda}[x_{n+1}+\ldots+x_{n+m}]$. This is made explicit using the combinatorics of non-attacking labellings as per HHL. For sufficiently large $m \geq |\mu|$ the $\mathbb{Q}(q,t)[x_1,\ldots,x_n]$-coefficients of the $m_{\lambda}[x_{n+1}+\ldots+x_{n+m}]$ stabilize to polynomials with coefficients that converge t-adically. 
\end{proof}

\begin{remark}
Note importantly, that for any composition $\mu$ and $m \geq 0$, by definition $\widetilde{E}_{\mu *0^m} = \widetilde{E}_{\mu}$.
\end{remark}

\subsubsection{Example}
Here we list a few simple examples.
\begin{itemize}
    \item $\widetilde{E}_{(1)} = x_1$
    \item $\widetilde{E}_{(2,0)} = x_1^2 + \frac{q^{-1}(1-t)}{1-q^{-1}t}x_1m_1[x_2+x_3+\cdots]$
    \item $\widetilde{E}_{(0,2)} = x_2^2 + (1-t)x_1^2 + \frac{1-q^{-1}t+q^{-1}}{1-q^{-1}t}(1-t)x_1x_2 + \left( \frac{q^{-1}(1-t)}{1-q^{-1}t}x_2 + \frac{q^{-1}(1-t)^2}{1-q^{-1}t}x_1 \right) m_1[x_3+\cdots] $
    \item $\widetilde{E}_{(2,2)} = x_1^2x_2^2 + \frac{q^{-1}(1-t)}{1-q^{-1}t}(x_1^2x_2+x_1x_2^2)m_1[x_3+x_4+\cdots] + \left( \frac{q^{-2}(1-t)^2(1+t)}{q^{-2}t^3-q^{-1}t^2-q^{-1}t+1}\right)x_1x_2m_{1,1}[x_3+x_4+\cdots]$
\end{itemize}

\section{$\mathscr{Y}$  Weight Basis of $\mathscr{P}_{as}^{+}$}

Given a family of commuting operators $\{ y_i : i \in I \}$ and a weight vector v we denote its weight by the function $\alpha: I \rightarrow \mathbb{Q}(q,t)$ such that $y_iv = \alpha(i)v.$ We sometimes denote $\alpha$ as $(\alpha_1,\alpha_2,\ldots).$

\subsection{The $\widetilde{E}_{\mu}$ are $\mathscr{Y}$ weight vectors}

In what follows, the classical spectral theory for non-symmetric Macdonald polynomials is used to demonstrate that the limit functions $\widetilde{E}_{\mu}$ are  $\mathscr{Y}$ weight vectors. The below lemma is a simple application of this classical theory and of basic properties of the t-adic topology on $\mathbb{Q}(q,t)$.

\begin{lem} \label{lem1}
For a composition $\mu$ with $\ell(\mu) = n$ define $\alpha_{\mu}^{(m)}$ to be the weight of $E_{\mu *0^m}$. Then in the $t$-adic topology on $\mathbb{Q}(q,t)$ the sequence $t^{n+m}\alpha_{\mu}^{(m)}(i)$ converges in m to some $\widetilde{\alpha}_{\mu}(i) \in \mathbb{Q}(q,t)$. In particular, $\widetilde{\alpha}_{\mu}(i) = 0$ for $i > n$ and for $1\leq i\leq n$ we have that $\widetilde{\alpha}_{\mu}(i) = 0$  exactly when $\mu_i = 0$.
\end{lem}

\begin{proof}[Proof:]
Take $\mu = (\mu_1,\ldots,\mu_n)$. From classic double affine Hecke algebra theory we have $\alpha_{\mu}^{(0)}(i) = q^{\mu_i}t^{1-\beta_{\mu}(i)}$ where 
$$\beta_{\mu}(i) := \#\{j: 1\leq j \leq i ~, \mu_j \leq \mu_i\} + \#\{j: i < j \leq n ~, \mu_i > \mu_j\}.$$
It follows then that

\[ t^{n+m}\alpha_{\mu}^{(m)}(i) = 
    \begin{cases}
    q^{\mu_i}t^{n+m+1-(\beta_{\mu}(i) + m \mathbbm{1}(\mu_i \neq 0))} = t^n\alpha_{\mu}^{(0)}(i) & i\leq n, \mu_i \neq 0\\
    q^{\mu_i}t^{n+m+1-(\beta_{\mu}(i) + m \mathbbm{1}(\mu_i \neq 0))} = t^{n+m}\alpha_{\mu}^{(0)}(i) & i\leq n, \mu_i = 0 \\
    t^{n+m+1-(\#(\mu_{j} = 0) + i-n)} = t^{\#(\mu_j \neq 0)}t^{m+1-(i-n)} & i > n
     \end{cases}
\]

Lastly, by limiting $m\rightarrow \infty$ we get the result.
\end{proof}

For a composition $\mu$ define the list of scalars $\widetilde{\alpha}_{\mu}$ using the formula in Lemma \ref{lem1} for $\widetilde{\alpha}_{\mu}(i)$ for $i \in \mathbb{N}$. We use Lemma \ref{lem1} to show that certain denominators that occur in the proof of Lemma \ref{lem2} below do not vanish in the limit as $m \rightarrow \infty $.

\begin{lem} \label{lem2}
For $\mu = (\mu_1,\ldots,\mu_n)$ with $\mu_i \neq 0$ for $1\leq i \leq n$, $\widetilde{E}_{\mu}$ is a $\y$-weight vector with weight $\widetilde{\alpha}_{\mu}$. 
\end{lem}

\begin{proof}
We spare the reader the direct calculation which uses the limit definition of the $\y_r$ operators and Prop. 6.21 from \cite{Ion_2022} which leads to

\begin{equation}\label{Y_r explicit}
    \y_r( \widetilde{E}_{\mu}) = \widetilde{\alpha}_{\mu}(r) (T_{r-1}\cdots T_1 \rho T_1^{-1}\cdots T_{r-1}^{-1}) \widetilde{E}_{\mu} .
\end{equation}
 
We will show that the right side of \eqref{Y_r explicit} is $\widetilde{\alpha}_{\mu}(r)\widetilde{E}_{\mu}$.
As $\widetilde{\alpha}_{\mu}(r) = 0$ for $r > n$ by Lemma \ref{lem1}, the lemma holds for $r \leq n$. Now let us consider some fixed $r\leq n$.
Below we show that $x_1 | T_1^{-1}\cdots T_{r-1}^{-1}\widetilde{E}_{\mu}$ from which it follows that $$\rho(T_1^{-1}\cdots T_{r-1}^{-1}\widetilde{E}_{\mu}) = T_1^{-1}\cdots T_{r-1}^{-1}\widetilde{E}_{\mu} $$
implying
\begin{align*}
\y_r(\widetilde{E}_{\mu}) &= \widetilde{\alpha}_{\mu}(r) (T_{r-1}\cdots T_1 \rho T_1^{-1}\cdots T_{r-1}^{-1}) \widetilde{E}_{\mu} \\
&= \widetilde{\alpha}_{\mu}(r) (T_{r-1}\cdots T_1 T_1^{-1}\cdots T_{r-1}^{-1}) \widetilde{E}_{\mu} \\
&= \widetilde{\alpha}_{\mu}(r)\widetilde{E}_{\mu}
\end{align*}
as desired.
To show that $x_1 | T_1^{-1}\cdots T_{r-1}^{-1}\widetilde{E}_{\mu}$ it suffices to show that for all $m \geq 0$,
$x_1 | T_1^{-1}\cdots T_{r-1}^{-1} E_{\mu *0^m}$. To this end fix $m \geq 0$. We have that 

\begin{align*}
\alpha_{\mu}^{(m)}(r)E_{\mu*0^m}&= Y_r^{(n+m)}(E_{\mu*0^m}) \\
&= t^{-(r-1)}T_{r-1}\cdots T_1 \omega_{n+m}^{-1} T_{n+m-1}^{-1}\cdots T_r^{-1} E_{\mu*0^m}.\\
\end{align*}

Since $\alpha_{\mu}^{(m)}(r) \neq 0$ we can have $\frac{1}{\alpha_{\mu}^{(m)}(r)} T_1 ^{-1}\cdots T_{r-1}^{-1}$ act on both sides to get 

\begin{align*}
T_1^{-1}\cdots T_{r-1}^{-1} E_{\mu*0^m} &= \frac{t^{-(r-1)}}{\alpha_{\mu}^{(m)}(r)}\omega_{n+m}^{-1} T_{n+m-1}^{-1}\cdots T_r^{-1} E_{\mu*0^m}.
\end{align*}
 By HHL any non-attacking labelling of $\mu* 0^m$ will have row 1 diagram labels given by $\{1,2,\ldots,n\}$ so $x_1\cdots x_n$ divides $E_{\mu*0^m}$ so in particular $x_r$ divides $E_{\mu*0^m}$ for all $m \geq 0$. Lastly, 
 \begin{align*}
\omega_{n+m}^{-1} T_{n+m-1}^{-1}\cdots T_r^{-1}X_r &= \omega_{n+m}^{-1} t^{-(n+m-r)}X_{n+m}T_{n+m-1}\cdots T_r \\
&= qt^{-(n+m-r)} X_1 \omega_{n+m}^{-1} T_{n+m-1}\cdots T_r\\
 \end{align*}
Thus $x_1$ divides $T_1^{-1}\cdots T_{r-1}^{-1} E_{\mu*0^m}$ for all $m \geq 0$ showing the result.
\end{proof}

Now we consider the general situation where the composition $\mu$ can have some parts which are 0. We can extend the above result, Lemma \ref{lem2}, by a straight-forward argument using intertwiner theory from the study of affine Hecke algebras.

\begin{thm} \label{thm5}
For all compositions $\mu$, $\widetilde{E}_{\mu}$ is a $\y$-weight vector with weight $\widetilde{\alpha}_{\mu}.$
\end{thm}

\begin{proof}[Proof Sketch:]
 Lemma \ref{lem2} shows that this statement holds for any composition with all parts nonzero. Further, every composition $\mu$ can be written as a permutation of a composition of the form $\nu * 0^m$ for a partition $\nu$ and some $m \geq 0$. Hence, it suffices to show that for any composition $\mu$, if $\widetilde{E}_{\mu}$ satisfies the theorem then so will $\widetilde{E}_{s_i(\mu)}$. This process is made rigorous by using induction on Bruhat order. Using the intertwiner operators from standard affine Hecke algebra theory, given by $\varphi_i = T_i\y_i-\y_iT_i$, we only need to show that for any $\mu$ with $s_i(\mu) > \mu$ in Bruhat order, $$\varphi_i \widetilde{E}_{\mu} = (\widetilde{\alpha}_{\mu}(i) - \widetilde{\alpha}_{\mu}(i+1))\widetilde{E}_{s_i(\mu)}.$$ Suppose the theorem holds for some $\mu$ with $\ell(\mu) = n$ and let $1 \leq i \leq n$ such that $s_i(\mu) > \mu$. Then we have the following:
\begin{align*}
    \varphi_{i}\widetilde{E}_{\mu} &= (T_i(\y_i - \y_{i+1}) + (1-t)\y_{i+1})\widetilde{E}_{\mu} \\
    &= (\widetilde{\alpha}_{\mu}(i)- \widetilde{\alpha}_{\mu}(i+1))T_i\widetilde{E}_{\mu} + (1-t)\widetilde{\alpha}_{\mu}(i+1)\widetilde{E}_{\mu} \\
    &= \lim_{m} (t^{n+m}\alpha_{\mu}^{(m)}(i)- t^{n+m}\alpha_{\mu}^{(m)}(i+1))T_iE_{\mu *0^m} + (1-t)t^{n+m}\alpha_{\mu}^{(m)}(i+1)E_{\mu *0^m} \\
    &= \lim_{m}(t^{n+m}\alpha_{\mu}^{(m)}(i)- t^{n+m}\alpha_{\mu}^{(m)}(i+1))E_{s_i(\mu) *0^m}\\
    &= (\widetilde{\alpha}_{\mu}(i) - \widetilde{\alpha}_{\mu}(i+1))\widetilde{E}_{s_i(\mu)}.\\
\end{align*}

\end{proof}

We have shown in Theorem \ref{thm5} there is an explicit collection of $\mathscr{Y}$-weight vectors $\widetilde{E}_{\mu}$ in $\mathscr{P}_{as}^{+}$ arising as the limits of non-symmetric Macdonald polynomials $E_{\mu * 0^m}$. Unfortunately, these $\widetilde{E}_{\mu}$ do not span $\mathscr{P}_{as}^{+}$. To see this note that one cannot write a non-constant symmetric function as a linear combination of the $\widetilde{E}_{\mu}$. However, in the below work we build a full $\y$ weight basis. 

\subsection{Constructing the Weight Basis}

 To complete our construction of a full weight basis of $\mathscr{P}_{as}^{+}$ one needs the $\partial_{-}^{(k)}$ operators from Ion and Wu. These operators are, up to a change of variables and plethsym, the $d_{-}$ operators from Carlson and Mellit's standard $\mathbb{A}_{q,t}$ representation. 

\begin{defn}\cite{Ion_2022} \label{defn6}
 Define the operator $\partial_{-}^{(k)}: \mathscr{P}(k)^{+} \rightarrow \mathscr{P}(k-1)^{+}$ to be the $\mathscr{P}_{k-1}^{+}$-linear map which acts on elements of the form $x_k^nF[x_{k+1}+x_{k+2}\cdots]$ for $F \in \Lambda$ and $n \geq 0$ as 
 $$\partial_{-}^{(k)}(x_k^nF[x_{k+1}+x_{k+2}+\cdots]) = \mathscr{B}_n(F)[x_k + x_{k+1} + \cdots]. $$
 Here the $\mathscr{B}_n$ are the Jing operators which serve as creation operators for the Hall-Littlewood symmetric functions $\mathcal{P}_{\lambda}$ given explicitly by the following plethystic formula:
 $$\mathscr{B}_n(F)[X] = \langle z^n \rangle  F[X-z^{-1}]Exp[(1-t)zX] .$$
 
\end{defn}
We refer the reader to \cite{Ion_2022} for a discussion on the Jing operators.
 Importantly, the $\partial_{-}^{(k)}$ operators do not come from the $\mathscr{H}^{+}$ action itself. Note that the $\partial_{-}^{(k)}$ operators are homogeneous by construction. 

We require the following lemma.

\begin{lem}\cite{Ion_2022} \label{lem3}
The map $\partial_{-}^{(n)}: \mathcal{P}(n)^{+} \rightarrow \mathcal{P}(n-1)^{+}$ is a projection onto $\mathcal{P}(n-1)^{+}$ i.e. for $f \in \mathcal{P}(n-1)^{+}\subset \mathcal{P}(n)^{+}$ we have that $\partial_{-}^{(n)}(f) = f$.
\end{lem}

Lemma \ref{lem3} shows that the following operator is well defined.

\begin{defn} \label{defn7}

For $f\in \mathscr{P}(n)^{+} \subset \mathscr{P}_{as}^{+}$  define $\widetilde{\sigma}(f) := \partial_{-}^{(1)}\cdots \partial_{-}^{(n)} f $. Then  $\widetilde{\sigma}$ defines an operator $\mathscr{P}_{as}^{+} \rightarrow \Lambda$ which we call the \textbf{\textit{stable-limit symmetrization operator}}. 
For a partition $\lambda$ define $\mathcal{A}_{\lambda} = \widetilde{\sigma}( \widetilde{E}_{\lambda}) \in \Lambda$.
\end{defn}

The $\mathcal{A}_{\lambda}$ symmetric functions have many useful properties including, but not limited to, the following.

\begin{thm} \label{thm6}
The set $\{ \mathcal{A}_{\lambda} : \lambda ~ \text{is a partition} \}$ is a basis of $\Lambda$.
\end{thm}

\begin{proof}[Proof Sketch]
The result follows after proving the stronger property that each $A_{\lambda}$ has a unitriangular expansion with respect to dominance order into the Hall-Littlewood symmetric function basis.
\end{proof}

Stable-limit symmetrization behaves well with respect to permuting the defining composition $\mu$ of each $\widetilde{E}_{\mu}$.

\begin{lem} \label{lem4}
For any composition $\mu$ there is some nonzero scalar $\gamma_{\mu} \in \mathbb{Q}(q,t)$ such that 
$$ \widetilde{\sigma}( \widetilde{E}_{\mu} ) = \gamma_{\mu} \mathcal{A}_{\text{sort}(\mu)} $$
where $\gamma_{\mu} = 1$ when $\mu$ is a partition.
\end{lem}

We can now construct a full $\y$-weight basis of $\mathscr{P}_{as}^{+}$. We parameterize this basis by pairs $(\mu | \lambda)$ for $\mu$ a composition and $\lambda$ a partition.

\begin{defn} \label{defn8}
For $\mu$ be a composition and $\lambda$ a partition define the \textbf{\textit{stable-limit non-symmetric Macdonald function}} corresponding to $(\mu|\lambda)$ as
$$\widetilde{E}_{(\mu|\lambda)} :=  \partial_{-}^{(\ell(\mu)+1)}\cdots \partial_{-}^{(\ell(\mu)+\ell(
\lambda))} \widetilde{E}_{\mu*\lambda}. $$
\end{defn}

\begin{remark}
\noindent Note importantly $\widetilde{E}_{(\mu|\lambda)} \in \mathscr{P}(\ell(\mu))^{+}$, $\widetilde{\sigma}(\widetilde{E}_{(\mu|\lambda)} ) = \widetilde{\sigma}(\widetilde{E}_{\mu*\lambda})$, and $\widetilde{E}_{(\mu|\lambda)}$ is homogeneous of degree $|\mu| + |\lambda|$. Further, for any composition $\mu$ and partition $\lambda$ we have $\widetilde{E}_{(\mu|\emptyset)} = \widetilde{E}_{\mu}  $ and $\widetilde{E}_{(\emptyset |\lambda)} = \mathcal{A}_{\lambda}$.
\end{remark}

The following simple lemma shows that the stable-limit non-symmetric Macdonald functions $\widetilde{E}_{(\mu|\lambda)}$ are $\y$-weight vectors.

\begin{lem} \label{lem5}
Suppose $f \in \mathcal{P}(k)^{+}$ is a $\y$-weight vector with weight $(\alpha_1,\ldots,\alpha_k,0,0,\ldots)$. Then $\partial_{-}^{(k)} f \in \mathcal{P}(k-1)^{+}$ is a $\y$-weight vector with weight $(\alpha_1,\ldots,\alpha_{k-1},0,0,\ldots)$. 
\end{lem}

\begin{proof}[Proof Sketch]
We know that for $g \in \mathcal{P}(k)^{+}$ and $1\leq i \leq k-1$, $\y_i \partial_{-}^{(k)} g = \partial_{-}^{(k)}\y_i g$ so $\y_i\partial_{-}^{(k)} f = \partial_{-}^{(k)}\y_i f = \alpha_i \partial_{-}^{(k)} f.$
One can show that if $i\geq k$ then $\y_i$ annihilates $\mathscr{P}(k-1)$.
Since $\partial_{-}^{(k)}f \in \mathcal{P}(k-1)^{+}$ for all $i \geq k$, $\y_i \partial_{-}^{(k)}f = 0$.

\end{proof}

Here we give a few basic examples of stable-limit non-symmetric Macdonald functions expanded in the Hall-Littlewood basis $\mathcal{P}_{\lambda}$ and their corresponding weights:
\begin{itemize}
    \item $\widetilde{E}_{(\emptyset|2)} = \mathcal{P}_{2}[x_1+\cdots] +\frac{q^{-1}}{1-q^{-1}t}\mathcal{P}_{1,1}[x_1+\cdots]$ and has weight $(0,0,\ldots)$
    \item $\widetilde{E}_{(0|2)} = \mathcal{P}_2[x_2+\cdots]+(1-t)x_1^2+\frac{q^{-1}}{1-q^{-1}t}\mathcal{P}_{1,1}[x_2+\cdots]+\frac{(1+q^{-1})(1-t)}{1-q^{-1}t}x_1\mathcal{P}_1[x_2+\cdots]$ and has weight $(0,q^2t,0,\ldots)$
    \item $\widetilde{E}_{(1|1,1)} = x_1\mathcal{P}_{1,1}[x_2+\cdots]$ and has weight $(qt^3,0,\ldots)$
\end{itemize}

Finally, we prove that the stable-limit non-symmetric Macdonald functions are a basis for $\mathscr{P}_{as}^{+}$.

\begin{thm}(Main Theorem) \label{thm7}
The $\widetilde{E}_{(\mu|\lambda)}$ are a $\y$-weight basis for $\mathcal{P}_{as}^{+}$.
\end{thm}

\begin{proof}[Proof Sketch]
As there are sufficiently many $\widetilde{E}_{(\mu|\lambda)}$ in each graded component of every $\mathcal{P}(k)^{+}$ it suffices to show that these functions are linearly independent. Obviously weight vectors in distinct weight spaces are linearly independent. Using Lemmas \ref{lem2} and \ref{lem5}, we deduce that if $\widetilde{E}_{(\mu_1|\lambda_1)}$ and $\widetilde{E}_{(\mu_2|\lambda_2)}$ have the same weight then necessarily $\mu_1 = \mu_2$. Hence, we can restrict to the case where we
have a dependence relation
$$ c_1\widetilde{E}_{(\mu|\lambda^{(1)})} +\cdots + c_N \widetilde{E}_{(\mu|\lambda^{(N)})} = 0 $$
for $\lambda^{(1)},\ldots,\lambda^{(N)}$ distinct partitions.
By applying the stable-limit symmetrization operator we see that 
$$ \widetilde{\sigma}(c_1\widetilde{E}_{(\mu|\lambda^{(1)})} +\cdots + c_N \widetilde{E}_{(\mu|\lambda^{(N)})} ) = \widetilde{\sigma}(c_1\widetilde{E}_{\mu*\lambda^{(1)}} +\cdots+c_N \widetilde{E}_{\mu*\lambda^{(N)}}) = 0 .$$
Now by Lemma \ref{lem4}, $\widetilde{\sigma}(\widetilde{E}_{\mu*\lambda^{(i)}}) = \gamma_{\mu*\lambda^{(i)}}  \mathcal{A}_{\text{sort}(\mu*\lambda^{(i)})} $ with nonzero scalars  $\gamma_{\mu*\lambda^{(i)}}$ so 
$$ 0 = c'_1\mathcal{A}_{\text{sort}(\mu*\lambda^{(1)})} +\ldots+ c'_n \mathcal{A}_{\text{sort}(\mu*\lambda^{(N)})}.$$ The partitions $\lambda^{(i)}$ are distinct so we know that the partitions $\text{sort}(\mu*\lambda^{(i)})$ are distinct as well. By Theorem \ref{thm6} the symmetric functions $\mathcal{A}_{\text{sort}(\mu*\lambda^{(i)})}$ are linearly independent. Thus $c'_i = 0$ implying $c_i = 0$ for all $1 
\leq i \leq N$ as desired.
\end{proof}

\printbibliography


\end{document}